\documentclass{amsart}
\usepackage{enumerate,amssymb,amsmath,graphics,epsfig}
\usepackage{color}
\usepackage{xspace}
\usepackage{mathrsfs}

\newcommand\RE{\mathbb{R}}

\newcommand\bfu{\mathbf{u}}
\newcommand\bfv{\mathbf{v}}
\newcommand\V{\boldsymbol{V}}
\newcommand\Q{\boldsymbol{Q}}
\newcommand\Qo{\boldsymbol{Q}^0}
\newcommand\W{W}
\newcommand\bfn{\mathbf{n}}
\newcommand\bft{\mathbf{t}}
\newcommand\bff{\mathbf{f}}
\newcommand\bfp{\mathbf{p}}
\newcommand\bfq{\mathbf{q}}
\newcommand\bfg{\mathbf{g}}
\newcommand\bfx{\mathbf{x}}
\newcommand\bfH{\boldsymbol{H}}
\newcommand\bfL{\boldsymbol{L}}
\newcommand\bfzero{\mathbf{0}}
\newcommand\sfA{\mathsf{A}}
\newcommand\sfB{\mathsf{B}}
\newcommand\sfC{\mathsf{C}}
\newcommand\sfD{\mathsf{D}}
\newcommand\sfx{\mathsf{x}}
\newcommand\sfy{\mathsf{y}}
\newcommand\sfzero{\mathsf{0}}
\newcommand\bfsigma{\boldsymbol{\sigma}}
\newcommand\bftau{\boldsymbol{\tau}}
\newcommand\grad{\operatorname{\textnormal{\bf grad}}}
\renewcommand\div{\operatorname{div}}
\newcommand\Hdiv{\boldsymbol{H}(\div;\Omega)}

\newcommand\Hodivom{\boldsymbol{H}_0(\div^0;\Omega;\mu)}
\newcommand\rot{\operatorname{rot}}
\newcommand\Hrot{\boldsymbol{H}(\rot;\Omega)}
\newcommand\Horot{\boldsymbol{H}_0(\rot;\Omega)}
\newcommand\curl{\operatorname{\textnormal{\bf curl}}}
\newcommand\Hcurl{\boldsymbol{H}(\curl)}

\newcommand\Hocurl{\boldsymbol{H}_0(\curl;\Omega)}
\newcommand\Hu{H^1(\Omega)}
\newcommand\Ldo{L^2_0(\Omega)}
\newcommand\Po{\mathcal{P}}
\newcommand\Ned{\mathcal{N}}
\newcommand\T{\mathcal{T}}
\newcommand\eps{\varepsilon}
\newcommand\Span{\mathrm{span}}

\newtheorem{theorem}{Theorem}
\newtheorem{proposition}{Proposition}
\newtheorem{corollary}{Corollary}
\newtheorem{remark}{Remark}

\begin{document}
\title[]{Approximation of the Maxwell eigenvalue problem in a Least-Squares
setting}
\author{Fleurianne Bertrand}
\address{Fakult\"at f\"ur Mathematik, TU Chemnitz, Germany}
\email{fleurianne.bertrand@mathematik.tu-chemnitz.de}
\urladdr{https://www.tu-chemnitz.de/mathematik/numapde/}
\author{Daniele Boffi}
\address{King Abdullah University of Science and Technology (KAUST), Saudi
Arabia and Universit\`a degli Studi di Pavia, Italy}
\email{daniele.boffi@kaust.edu.sa}
\urladdr{https://cemse.kaust.edu.sa/people/person/daniele-boffi}
\author{Lucia Gastaldi}
\address{DICATAM, Universit\`a degli Studi di Brescia, Italy and IMATI-CNR,
Pavia, Italy}
\email{lucia.gastaldi@unibs.it}
\urladdr{http://lucia-gastaldi.unibs.it}

\begin{abstract}
We discuss the approximation of the eigensolutions associated with the Maxwell
eigenvalues problem in the framework of least-squares finite elements. We
write the Maxwell curl curl equation as a system of two first order equation
and design a novel least-squares formulation whose minimum is attained at
the solution of the system. The eigensolution are then approximated by
considering the eigenmodes of the underlying solution operator. We study the
convergence of the finite element approximation and we show several numerical
tests confirming the good behavior of the method. It turns out that nodal
elements can be successfully employed for the approximation of our problem
also in presence of singular solutions.
\end{abstract}
\maketitle
\section{Introduction}
\label{se:intro}

The finite element approximation of the eigenmodes associated with the Maxwell
system is a deeply studied and nowadays well understood topic. In
particular, it is universally recognized that the natural choice for the
approximation of the eigensolutions associated with the curl curl operator, is
to consider N\'ed\'elec (edge) finite elements~\cite{nedelec}. Approximations
based on edge elements are optimally convergent, do not present any spurious
modes, and are robust in presence of singularities due to the domain or to the
presence of different materials. The interested reader is referred to the
related literature; in particular to~\cite{bossavit3} for a discussion about
Whitney forms in connection with this problem, to~\cite{bfgp,boffifortid} for
the first analysis of the spectral correctness of edge elements, and
to~\cite{hiptacta,monk,acta} for general surveys on this subject.

Formulations based on least-squares finite elements are widely used for the
approximation of models involving partial differential
equations~\cite{bochev}. Recent studies are investigating the behavior of the
spectrum of operators associated with least-squares finite element
formulations~\cite{bbLS,linda}. These studies have their interest by
themselves, and in some cases they can contribute to the design of new
schemes.

In this paper we begin the study of the eigenvalues associated with the
Maxwell system in the framework of least-squares finite elements. The
interest of this research is twofold: on one side we discuss how to introduce
a first order formulation of the Maxwell system in this framework, on the
other side, we analyze rigorously the approximation of the eigensolutions with
various choices of finite element spaces in two and three dimensions.
Besides formulations based on edge elements, we believe that a remarkable
result of our investigation is that standard Lagrangian (nodal) elements can
be successfully used in two dimensions and, with some care, in three
dimensions. In two dimensions, the use of nodal elements is supported by a
rigorous theory, valid when the solution satisfies appropriate regularity
assumptions. Our numerical investigations show that actually the approximation
based on nodal elements is performing well also in presence of strong
singularities, such as those arising from reentrant corners, cracks, and
material discontinuities.

In three dimensions, the theory covers the case of edge elements, while the
nodal element approximation requires a more specific analysis, probably
depending on the structure of the mesh, which will be the object of future
investigations.

Several numerical experiments complement the theoretical results, confirming
the theory and supplementing the theoretical investigations when they are not
available.

Section~\ref{se:setting} describes the problem we are dealing with and
introduced the first order least-squares formulation. We continue then in
Section~\ref{se:2D} with the discussion of the two dimensional case. Indeed,
the two and three dimensional case, although sharing some analogies, are
intrinsically different: in two dimensions the problem is equivalent to the
Laplace eigenvalue problem with Neumann boundary conditions.
Section~\ref{se:2Dnum} is devoted to the numerical approximation of the two
dimensional problem. Particular care is devoted to the definition of the
solutions of our generalized eigenvalue problem and to the description of
possible degenerate situations. Several two dimensional numerical results are
presented in Section~\ref{se:2Dres}, confirming the good behavior of nodal
element approximations also in presence of strong singularities. The last to
sections deal with the three dimensional case: in Section~\ref{se:3D} the
theory is developed, while in Section~\ref{se:3Dres} some numerical results
are presented.

\section{Problem setting}
\label{se:setting}

Let $\Omega$ be a domain in $\RE^3$. We start by the case of a contractible
domain $\Omega$ where harmonic forms do not enter the characterization of our
solutions.
We are interested in the following eigenvalue problem associated with
Maxwell's equation: find $\lambda\in\RE$ and a non-vanishing $\bfu$ such that
\begin{equation}
\left\{
\aligned
&\curl(\mu^{-1}\curl\bfu)=\lambda\eps\bfu&&\text{in }\Omega\\
&\div(\eps\bfu)=0&&\text{in }\Omega\\
&\bfu\times\bfn=0&&\text{on }\partial\Omega,
\endaligned
\right.
\label{eq:pb3D}
\end{equation}
where $\bfn$ is the outward normal unit vector to the boundary of the domain
$\Omega$, and where $\mu$ and $\eps$ are the (possibly varying) magnetic
permeability and electric permittivity, respectively.

The source problem corresponding to~\eqref{eq:pb3D} reads: given $\bff$ with
$\div(\eps\bff)=0$, find $\bfu$ such that
\begin{equation}
\left\{
\aligned
&\curl(\mu^{-1}\curl\bfu)=\eps\bff&&\text{in }\Omega\\
&\div(\eps\bfu)=0&&\text{in }\Omega\\
&\bfu\times\bfn=0&&\text{on }\partial\Omega.
\endaligned
\right.
\label{eq:pbf3D}
\end{equation}

A representation of~\eqref{eq:pbf3D} as a system of first order equations
could be done, in analogy to the usual procedure for the Laplace equation, as
follows by introducing the auxiliary variable $\bfsigma=\mu^{-1}\curl\bfu$, so
that we have
\[
\left\{
\aligned
&\bfsigma=\mu^{-1}\curl\bfu&&\text{in }\Omega\\
&\curl\bfsigma=\eps\bff&&\text{in }\Omega\\
&\div(\eps\bfu)=0&&\text{in }\Omega\\
&\bfu\times\bfn=0&&\text{on }\partial\Omega.
\endaligned
\right.
\]
Unfortunately, this system is not suitable to be approximated by a
Least-Squares finite element strategy on non-smooth domains when a
minimization principle in $L^2(\Omega)$ is used. Indeed, the functional to be
minimized would read
\[
\mathcal{F}(\bftau,\bfv)=
\|\bftau-\mu^{-1}\curl\bfv\|_0^2+\|\curl\bftau-\eps\bff\|_0^2
+\|\div(\eps\bfv)\|_0^2.
\]
It is apparent that no reasonable choice of functional spaces can be made in
this situation.
Already in the simpler case when $\eps\equiv1$, this would imply that the
variable $\bfv$ should have both divergence and curl
bounded in $L^2(\Omega)$. This is a well known source of troubles for the
finite element approximation when the domain has non convex corners or
edges~\cite{costabel91,fix-stephan82,cox-fix84}, since in that case singular
solutions $\bfu$ are not in $\boldsymbol{H}^1(\Omega)$.

For this reason, we make use of the first order system introduced
in~\cite{bfgp}. Namely we consider a vectorfield $\bfg$ such that
$\curl\bfg=\eps\bff$, $\div(\mu\bfg)=0$, $(\mu\bfg)\cdot\bfn=0$ on
$\partial\Omega$, and look for the pair $(\bfu,\bfp)$
satisfying
\begin{equation}
\left\{
\aligned
&\eps\bfu=\curl\bfp&&\text{in }\Omega\\
&\mu^{-1}\curl\bfu=\bfg&&\text{in }\Omega\\
&\bfu\times\bfn=0&&\text{on }\partial\Omega.
\endaligned
\right.
\label{eq:bfgp3D}
\end{equation}
In general $\bfp$ is not unique, but there is only one $\bfp$ that satisfies
the additional conditions
\begin{equation}
\left\{
\aligned
&\div(\mu\bfp)=0&&\text{in }\Omega\\
&(\mu\bfp)\cdot\bfn=0&&\text{on }\partial\Omega.
\endaligned
\right.
\label{eq:gauge3D}
\end{equation}
We assume $\eps$ and $\mu$ to be real scalar functions satisfying 
\begin{equation}
\label{eq:epsmu}
0<\underline\eps\le\eps\le\overline\eps,\qquad
0<\underline\mu\le\mu\le\overline\mu
\end{equation}
for almost every $\bfx$ in $\Omega$ and introduce the following spaces
\[
\aligned
&\Hcurl=\{\bfv\in\bfL^2(\Omega):\curl\bfv\in\bfL^2(\Omega)\}\\
&\Hocurl=\{\bfv\in\Hcurl:\bfv\times\bfn=0\text{ on }\partial\Omega\}\\
&\Hodivom=\{\bfq\in\bfL^2(\Omega):\div(\mu\bfq)=0\text{ in }\Omega,\
(\mu\bfq)\cdot\bfn=0\text{ on }\partial\Omega\}.
\endaligned
\]
Therefore, $\bfg\in\Hodivom$ and we look for a solution
$\bfp\in\Hodivom$.

\begin{proposition}
Let us consider Problems~\eqref{eq:pbf3D} and~\eqref{eq:bfgp3D} with
\[
\left\{
\aligned
&\curl\bfg=\eps\bff&&\text{in }\Omega\\
&\div(\mu\bfg)=0&&\text{in }\Omega\\
&(\mu\bfg)\cdot\bfn=0&&\text{on }\partial\Omega.
\endaligned
\right.
\]
If $\bfu$ solves~\eqref{eq:pbf3D} then there exists $\bfp$ so that $(\bfu,\bfp)$
solves~\eqref{eq:bfgp3D}. Conversely, if $\bfu$ solves~\eqref{eq:bfgp3D} then
it is also a solution of~\eqref{eq:pbf3D}.
\label{pr:1}
\end{proposition}

\begin{proof}

If $\bfu$ solves~\eqref{eq:pbf3D}, then from $\div(\eps\bfu)=0$ we get that
there exists $\bfp$ such that $\eps\bfu=\curl\bfp$. Such $\bfp$ is defined up
to an additive gradient that can be chosen such that~\eqref{eq:gauge3D} is
satisfied.
Then, from the first equation in~\eqref{eq:pbf3D}, we have
$\curl(\mu^{-1}\curl\bfu-\bfg)=0$ which implies that
$\mu^{-1}\curl\bfu-\bfg=\grad\phi$ for some $\phi$ in $\Hu$.
On the other hand, from $\div(\mu\bfg)=0$ it follows $\div(\mu\grad\phi)=0$ in
$\Omega$, and the boundary conditions on $\bfu$ and $\bfg$ imply
$(\mu\grad\phi)\cdot\bfn=0$ on $\partial\Omega$; hence $\phi$ is constant,
from which we conclude that $\mu^{-1}\curl\bfu-\bfg=0$.

Conversely, taking the $\curl$ of the second equation in~\eqref{eq:bfgp3D}, we
get that $\bfu$ solves~\eqref{eq:pbf3D}.
The divergence free condition follows from $\eps\bfu=\curl\bfp$.

\end{proof}

In the spirit of Least-Squares formulation, this equivalence leads to the
minimization of the following functional
\begin{equation}
\mathcal{F}(\bfv,\bfq)=\|\eps^{1/2}\bfv-\eps^{-1/2}\curl\bfq\|_0^2
+\|\mu^{-1/2}\curl\bfv-\mu^{1/2}\bfg\|_0^2
\label{eq:LSbfgp}
\end{equation}
in the energy space $\Hocurl\times\Hcurl$, where we split $\eps$ and $\mu$ as
the square of their square roots in order to get a symmetric system when
considering the gradient of $\mathcal{F}$.

As mentioned above, the uniqueness of $\bfp$ requires the additional
conditions stated in~\eqref{eq:gauge3D}. This could be enforced by changing
the energy space for the minimization of~\eqref{eq:LSbfgp} to
$\Hocurl\times(\Hcurl\cap\Hodivom)$. Clearly, this would lead to the same
troubles described before related to the approximation of the space
$\Hcurl\cap\Hodivom$.
We postpone the discussion about this issue and we start our investigations
considering the two dimensional counterpart of~\eqref{eq:pb3D}.

\section{Problem setting in two dimensions}
\label{se:2D}

Let $\Omega$ be a polygonal domain in $\RE^2$. The eigenvalue problem we are
interested in, seeks for $\lambda\in\RE$ and a non-vanishing $\bfu$ such that
\begin{equation}
\left\{
\aligned
&\curl(\mu^{-1}\rot\bfu)=\lambda\eps\bfu&&\text{in }\Omega\\
&\div(\eps\bfu)=0&&\text{in }\Omega\\
&\bfu\cdot\bft=0&&\text{on }\partial\Omega,
\endaligned
\right.
\label{eq:pb2D}
\end{equation}
where $\bft$ is the counterclockwise tangent unit vector to the boundary of
the domain $\Omega$.
We recall the two dimensional definitions of the $\curl$ and $\rot$ operators:
\[
\rot\bfv=\frac{\partial v_2}{\partial x}-\frac{\partial v_1}{\partial
y}\qquad\text{with }\bfv=(v_1,v_2)^\top
\]
and
\[
\curl\varphi=\left(\frac{\partial\varphi}{\partial y},
-\frac{\partial\varphi}{\partial x}\right)^\top.
\]
For completeness, we also recall the integration by parts formula that
involves these operators
\[
\int_\Omega\curl(\mu^{-1}\rot\bfu)\cdot\bfv\,d\bfx
=\int_\Omega\mu^{-1}\rot\bfu\rot\bfv\,d\bfx
-\int_{\partial\Omega}\mu^{-1}\rot\bfu\ \bfv\cdot\bft\,ds
\]
which is valid whenever the involved integrals are finite.

The source problem corresponding to~\eqref{eq:pb2D} reads: given $\bff$ with
$\div(\eps\bff)=0$, find $\bfu$ such that
\begin{equation}
\left\{
\aligned
&\curl(\mu^{-1}\rot\bfu)=\eps\bff&&\text{in }\Omega\\
&\div(\eps\bfu)=0&&\text{in }\Omega\\
&\bfu\cdot\bft=0&&\text{on }\partial\Omega
\endaligned
\right.
\label{eq:pbf2D}
\end{equation}
and the two dimensional version of the first order system~\eqref{eq:bfgp3D}
is: find $\bfu$ and $p$ such that
\begin{equation}
\left\{
\aligned
&\eps\bfu=\curl p&&\text{in }\Omega\\
&\mu^{-1}\rot\bfu=g&&\text{in }\Omega\\
&\bfu\cdot\bft=0&&\text{on }\partial\Omega
\endaligned
\right.
\label{eq:bfgp2D}
\end{equation}
with $\mu g\in\Ldo$ such that $\curl g=\eps\bff$, where $\Ldo$ is the subspace of
$L^2(\Omega)$ of zero mean valued functions.  In this case the uniqueness of
$p$ is guaranteed by the condition
\[
\int_\Omega\mu p\,d\bfx=0.
\]

The following proposition is the analogue of Proposition~\ref{pr:1}. We recall
it here in the two dimensional setting for the reader's convenience.

\begin{proposition}
Let us consider problems~\eqref{eq:pbf2D} and~\eqref{eq:bfgp2D} with
$\eps\bff=\curl g$ and $\int_\Omega\mu g=0$.
If $\bfu$ solves~\eqref{eq:pbf2D}, then there exists $p$ such that $(\bfu,p)$
is solution of~\eqref{eq:bfgp2D}. Conversely, if $(\bfu,p)$ is solution
of~\eqref{eq:bfgp2D} then $\bfu$ solves~\eqref{eq:pbf2D}.
\end{proposition}

\begin{proof}

If $\bfu$ solves~\eqref{eq:pbf2D}, then from $\div(\eps\bfu)=0$ we get that
there exists $p$ such that $\eps\bfu=\curl p$. Since $p$ is defined up to an
additive constant, we choose it such that the mean value of $\mu p$ is zero on
$\Omega$.
Then, from the first equation in~\eqref{eq:pbf2D}, we have
$\curl(\mu^{-1}\rot\bfu-g)=0$ which implies that $\mu^{-1}\rot\bfu-g$ is
constant in $\Omega$.
On the other hand, the boundary conditions on $\bfu$ imply that the average of
$\rot\bfu$ is zero; hence the average of $\rot\bfu-\mu g$ is zero,
from which we conclude that $\rot\bfu-\mu g=0$.

Conversely, taking the $\curl$ of the second equation in~\eqref{eq:bfgp2D}, we
get that $\bfu$ solves~\eqref{eq:pbf2D}.
The divergence free condition follows from $\eps\bfu=\curl p$.

\end{proof}

It is well known that in two dimensions the Maxwell system we are considering,
is equivalent to a Neumann problem for the Laplace equation.

\begin{proposition}
The component $p$ of the solution of~\eqref{eq:bfgp2D} is the solution of the
following equation
\[
\left\{
\aligned
&\rot(\eps^{-1}\curl  p)=\mu g&&\text{in }\Omega\\
&\frac{\partial p}{\partial\bfn}=0&&\text{on }\partial\Omega.
\endaligned
\right.
\]
\label{pr:neumann}
\end{proposition}

\begin{remark}
\label{re:neumann}
A natural question, that will be essential for the analysis of the
discretization, is the regularity of the solution of~\eqref{eq:bfgp2D}. Thanks
to Proposition~\ref{pr:neumann}, we can discuss first the regularity of $p$
and then consider that $\eps\bfu=\curl p$. Clearly, the regularity of $p$
depends on the regularity of $g$, $\eps$, and $\mu$, and of the domain
$\Omega$. In general, it is well known that if $\Omega$ is a polygon then
there exists $s\in(1/2,1)$ such that $p\in H^{1+s}(\Omega)$ whenever
$\mu g\in\Ldo$. Since $\bfu=\eps^{-1}\curl p$ we have that 
$\bfu\in H^s(\Omega)$ and $\rot\bfu=\mu\bfg$. Moreover the following a priori
estimate hold true
\[
\|\bfu\|_s+\|p\|_{1+s}\le C\|g\|_0.
\]

Even for smoother $g$, $\eps$, and $\mu$, there are domains
where the regularity of $p$ is not higher. For instance, if $\Omega$ is the
L-shaped domain, then $s$ cannot be taken in general larger than or equal to
$2/3$.

For the analysis of the convergence of the eigenvalue problem, we are going
to consider $g\in\Hu$. For $\eps$ and $\mu$ smooth and nonsingular domains, we
have in this case $p\in H^3(\Omega)$.

\end{remark}

In the framework of Least-Squares finite elements, we are then led to the
minimization of the functional
\begin{equation}
\mathcal{F}(\bfv,q)=\|\eps^{1/2}\bfv -\eps^{-1/2}\curl q\|_0^2
+\|\mu^{-1/2}\rot\bfv-\mu^{1/2}g\|_0^2
\label{eq:LSbfgp2D}
\end{equation}
in the space $\V\times Q$, where $\V$ and $Q$  are defined as follows
\[
\aligned
&\V=\Horot\\
&Q=\{q\in H^1(\Omega):\mu q\in\Ldo\}
\endaligned
\]
and are equipped by the norm induced by the following scalar products
\[
\aligned
&(\bfu,\bfv)_{\V}=(\eps\bfu,\bfv)+(\mu^{-1}\rot\bfu,\rot\bfv)\\
&(p,q)_Q=(\eps^{-1}\curl p,\curl q).
\endaligned
\]
From the assumptions~\eqref{eq:epsmu} on $\eps$ and $\mu$, and from the
Poincar\'e inequality, the induced norms $\|\cdot\|_{\V}$ and $\|\cdot\|_Q$
are equivalent to the standard ones.

A variational formulation of~\eqref{eq:LSbfgp2D} is given by: find $\bfu\in\V$
and $p\in Q$ such that
\begin{equation}
\left\{
\aligned
&(\eps\bfu,\bfv)+(\mu^{-1}\rot\bfu,\rot\bfv)-(\bfv,\curl p)=
(g,\rot\bfv)&&\forall\bfv\in\V\\
&-(\bfu,\curl q)+(\eps^{-1}\curl p,\curl q)=0&&\forall q\in Q.
\endaligned
\right.
\label{eq:LSbfgpvar2D}
\end{equation}

The next proposition states the ellipticity of the bilinear form associated
with the above problem.

\begin{proposition}

Let
\[
a:\left(\V\times Q\right)\times\left(\V\times Q\right)\to\RE
\]
be the bilinear form associated with the
formulation~\eqref{eq:LSbfgpvar2D}, that is
\[
a(\bfu,p;\bfv,q)=(\mu^{-1}\rot\bfu,\rot\bfv)
+(\eps^{1/2}\bfu-\eps^{-1/2}\curl p,\eps^{1/2}\bfv-\eps^{-1/2}\curl q).
\]
Then there exists $\alpha>0$ such that
\[
a(\bfv,q;\bfv,q)\ge\alpha\left(\|\bfv\|^2_{\V}+\|q\|^2_Q\right).
\]
\label{pr:ell2D}
\end{proposition}

\begin{proof}

We start observing that
\[
a(\bfv,q;\bfv,q)=\|\mu^{-1/2}\rot\bfv\|_0^2
+\|\eps^{1/2}\bfv-\eps^{-1/2}\curl q\|_0^2\ge\|\mu^{-1/2}\rot\bfv\|_0^2.
\]

For a positive $\beta$ we have
\[
\aligned
a(\bfv,q;\bfv,q)&=(\mu^{-1}\rot\bfv,\rot\bfv)+(\eps\bfv,\bfv)-2(\bfv,\curl q)
+(\eps^{-1}\curl q,\curl q)\\
&\quad+2\beta(\bfv,\curl q)-2\beta(\mu^{-1/2}\rot\bfv,\mu^{1/2}q)
\pm\beta^2(\mu q,q)\\
&=\|\mu^{-1/2}\rot\bfv-\beta\mu^{1/2}q\|_0^2+(\eps\bfv,\bfv)
-2(1-\beta)(\bfv,\curl q)\\
&\quad+(\eps^{-1}\curl q,\curl q)-\beta^2(\mu q,q)
\pm(1-\beta)^2(\eps^{-1}\curl q,\curl q)\\
&=\|\mu^{-1/2}\rot\bfv-\beta\mu^{1/2}q\|_0^2
+\|\eps^{1/2}\bfv-(1-\beta)\eps^{-1/2}\curl q\|_0^2\\
&\quad-\beta^2(\mu q,q)+(2\beta-\beta^2)\|q\|_Q^2\\
&\ge-\beta^2\|\mu^{-1/2}q\|_0^2+(2\beta-\beta^2)\|q\|_Q^2.
\endaligned
\]
Using the Poincar\'e inequality $\|q\|_0\le C_P\|\curl q\|_0$ and the bounds
in~\eqref{eq:epsmu} we get
\[
a(\bfv,q;\bfv,q)\ge
(2\beta-\beta^2(1+\underline\mu^{-1}\overline\eps C_P^2))\|q\|_Q^2
\]
which for $\beta$ small enough gives
\[
a(\bfv,q;\bfv,q)\ge C_1\|q\|_Q^2.
\]
Finally, we estimate $\|\eps^{1/2}\bfv\|_0$. We have
\[
\|\eps^{1/2}\bfv\|_0\le\|\eps^{1/2}\bfv-\eps^{-1/2}\curl q\|_0
+\|\eps^{-1/2}\curl q\|_0
\]
from which we obtain
\[
\aligned
\|\eps^{1/2}\bfv\|^2_0&\le2\left(
\|\eps^{1/2}\bfv-\eps^{-1/2}\curl q\|^2_0+\|q\|^2_Q\right)\\
&\le 2(1+1/C_1)a(\bfv,q;\bfv,q).
\endaligned
\]

In conclusion, we have the ellipticity result with
\[
\alpha=\frac13\min\left(1,C_1,\frac{1}{2(1+1/C_1)}\right).
\]

\end{proof}

By using the Lax--Milgram lemma, we have existence and uniqueness of the
solution of~\eqref{eq:LSbfgpvar2D}.

\begin{corollary}

Given $g\in\Ldo$ there exists one and only one solution
of~\eqref{eq:LSbfgpvar2D} which satisfies the a priori stability bound
\[
\|\bfu\|_{\V}+\|p\|_Q\le C\|g\|_0.
\]

\end{corollary}

We introduce the solution operator $T:Q\to Q$ defined as follows: given $g\in Q$,
\[
Tg=p,
\]
where $p\in Q$ is the second component of the solution
of~\eqref{eq:LSbfgpvar2D}. The regularity stated in Remark~\ref{re:neumann}
implies that $T$ is compact. The following proposition states that $T$ is
self-adjoint.

\begin{proposition}

For all $g_1$ and $g_2$ in $Q$ it holds
\[
(Tg_1,g_2)_Q=(g_1,Tg_2)_Q.
\]
\label{pr:self-adjoint}
\end{proposition}

\begin{proof}

Let us denote by $\bfu_i\in\V$ the other component of the solution associated
with $Tg_i$ ($i=1,2$). Then, from the definition of $T$ and of the scalar
products in $\V$ and $Q$, it follows
\[
\aligned
(Tg_1,g_2)_Q&=(\eps^{-1}\curl(Tg_1),\curl g_2)=(\bfu_1,\curl g_2)
=(\rot\bfu_1,g_2)\\
&=(\bfu_2,\bfu_1)_{\V}-(\bfu_1,\curl(Tg_2))\\
&=(\bfu_2,\bfu_1)_{\V}-(\eps^{-1}\curl(Tg_1),\curl(Tg_2))\\
&=(\bfu_1,\bfu_2)_{\V}-(\eps^{-1}\curl(Tg_2),\curl(Tg_1))\\
&=(\bfu_1,\bfu_2)_{\V}-(\bfu_2,\curl(Tg_1))\\
&=(\eps^{-1}\curl(Tg_2),\curl g_1)=(g_1,Tg_2)_Q.
\endaligned
\]

\end{proof}

In analogy to what has been done in the case of the Laplace eigenvalue problem
in~\cite{bbLS}, it is then natural to consider the following variational
formulation in order to describe the solutions of the eigenvalue
problem~\eqref{eq:pb2D}. Find $\lambda\in\RE$ and $p\in Q$ with
$p\ne0$ such that for some $\bfu\in\V$ it holds
\begin{equation}
\left\{
\aligned
&(\eps\bfu,\bfv)+(\mu^{-1}\rot\bfu,\rot\bfv)-(\bfv,\curl p)=
\lambda(p,\rot\bfv)&&\forall\bfv\in\V\\
&-(\bfu,\curl q)+(\eps^{-1}\curl p,\curl q)=0&&\forall q\in Q.
\endaligned
\right.
\label{eq:LSbfgpvar2Deig}
\end{equation}

In Problem~\eqref{eq:LSbfgpvar2Deig} we look for real eigenvalues and use real
functional spaces. This is justified by the fact that the underlying operator
is self-adjoint (see Proposition~\ref{pr:self-adjoint}). Moreover, we have the
following orthogonality properties.

\begin{proposition}

Let $\lambda_i\ne\lambda_j$ be two eigenvalues of~\eqref{eq:LSbfgpvar2Deig},
and $(p_i,\bfu_i)$ and $(p_j,\bfu_j)$ the corresponding eigenfunctions. Then
the following orthogonalities are satisfied
\[
\aligned
&(p_i,p_j)_Q=(\eps^{-1}\curl(p_i),\curl(p_j))=0\\
&(\bfu_i,\bfu_j)_{\V}=(\eps\bfu_i,\bfu_j)+(\mu^{-1}\rot(\bfu_i),\rot(\bfu_j))=0.
\endaligned
\]

In case of multiple eigenvalues, the corresponding eigenfunctions can be
chosen so that the same orthogonalities are satisfied.

\end{proposition}

\begin{proof}

The result follows in a standard way by testing~\eqref{eq:LSbfgpvar2Deig} with
$\bfv=\bfu_j$ and $q=p_j$ when $\bfu=\bfu_i$, $p=p_i$, and
$\lambda=\lambda_i$. The same equation~\eqref{eq:LSbfgpvar2Deig} with the
roles of $i$ and $j$ swapped,
gives
\[
(\lambda_i+1)(p_i,\rot\bfu_j)=(\lambda_j+1)(p_j,\rot\bfu_i)
\]
which gives the results, observing that
\[
(p_i,\rot\bfu_j)=(\eps^{-1}\curl(p_i),\curl(p_j))=(p_j,\rot\bfu_i).
\]
When $\lambda_i$ is different from $\lambda_j$ this implies the first
orthogonality. The second one follows by inserting the first one into
Equation~\eqref{eq:LSbfgpvar2Deig}.

\end{proof}

\begin{remark}

From the above orthogonalities, it follows by standard arguments that
$Q=\Span\{p_i,\ i=1,\dots\}$. Moreover, from the second equation
of~\eqref{eq:LSbfgpvar2Deig} we have that
$\bfu_i=\eps^{-1}\curl p_i+\grad\phi_i$. Substituting $\bfu_i$ in the first
equation and taking $\bfv=\grad\phi_i$ gives $(\eps\grad\phi_i,\grad\phi_i)=0$
which implies $\grad\phi_i=\mathbf{0}$. It follows that the set of the
$\bfu_i$'s generates the subspace of $\V$ containing the vectorfields $\bfv$
with $\div(\eps\bfv)=0$.
\label{re:Remark}
\end{remark}

\subsection{On the structure of the spectrum}
\label{se:operatorform}

Eigenvalue problems in the form of~\eqref{eq:LSbfgpvar2Deig} are not standard
and, to the best of our knowledge, have been mainly used when discussing the
spectrum of operators arising from the Least-Squares finite element method.

Problem~\eqref{eq:LSbfgpvar2Deig} can be written as follows in terms of
operators:
\begin{equation}
\begin{pmatrix}
\sfA&\sfB^\top\\
\sfB&\sfC
\end{pmatrix}
\begin{pmatrix}
\sfx\\
\sfy
\end{pmatrix}
=\lambda
\begin{pmatrix}
\sfzero&\sfD\\
\sfzero&\sfzero
\end{pmatrix}
\begin{pmatrix}
\sfx\\
\sfy
\end{pmatrix}.
\label{eq:operatorform}
\end{equation}
The aim of this subsection is to collect some results about the structure of
the eigensolutions, discussing in particular the consequences of the
possible degeneracy of the right hand side of~\eqref{eq:operatorform}.

First of all, even if the problem does not seem symmetric, it originates
from~\eqref{eq:pb2D} which is associated with a self-adjoint solution
operator. Indeed, after observing that $\sfD=-\sfB^\top$, we can argue as
in~\cite{bbLS} to show that~\eqref{eq:operatorform} is equivalent to the
symmetric problem
\[
\begin{pmatrix}
\sfA&\sfzero\\
\sfzero&\sfzero
\end{pmatrix}
\begin{pmatrix}
\sfx\\
\sfy
\end{pmatrix}
=(\lambda+1)
\begin{pmatrix}
\sfzero&-\sfB^\top\\
-\sfB&-\sfC
\end{pmatrix}
\begin{pmatrix}
\sfx\\
\sfy
\end{pmatrix}
\]
and to the symmetric Schur complement formulations
\begin{gather*}
\sfA\sfx=(\lambda+1)\sfB^\top\sfC^{-1}\sfB\sfx\\
\sfC\sfy=(\lambda+1)\sfB\sfA^{-1}\sfB^\top\sfy.
\end{gather*}

We now proceed with some comments that are particularly important in view of
the numerical approximation, and of the three-dimensional extension.

Let us assume that $(\sfx,\sfy)^\top$ is such that the right hand side
of~\eqref{eq:operatorform} is vanishing. Then, $\sfB^\top\sfy=\sfzero$ which
implies $\curl p=0$, that is $p=0$ due to zero mean value condition on $p$.
This is not admissible, since Problem~\eqref{eq:LSbfgpvar2Deig} seeks a non
vanishing $p$.
On the other hand, the numerical approximation of~\eqref{eq:operatorform} will
have the analogous form of a generalized algebraic eigenvalue problem
involving matrices and vectors that we denote with the same symbols. It
happens that, if we follow the same argument as before (see also
Remark~\ref{re:Remark}), we may have solutions that correspond to a generic
$\bfx$ and to $\sfy=\sfzero$. These solutions correspond to eigenvalues
$\lambda=\infty$ that should be discarded in view of the condition $p\ne0$.
This should be taken into account when the numerical results are performed.

It is also interesting to observe what happens if we relax the zero mean value
condition on the space $Q$. In such case there exists a non vanishing $\sfy$
such that $\sfB^\top\sfy=-\sfD\sfy=\sfzero$. Such $\sfy$ corresponds to a
$\curl$-free $p$, that is $p$ constant. If we combine
$\sfy\in\ker(\sfB^\top)=\ker(\sfD)$ with a vanishing $\sfx$, we can see a
singular behavior of~\eqref{eq:operatorform} that can be summarized by the
equation
\[
\sfzero=\lambda\sfzero,
\]
that is, $\lambda$ cannot be determined. Notice, that this singular behavior
doesn't occur if, for instance, the zero mean value condition is dropped in
the case of the standard Galerkin approximation of the Neumann Laplace
eigenproblem: in such case the original spectrum is modified by adding a
vanishing eigenvalue that corresponds to the constant eigenfunction.

\section{Two dimensional finite element analysis}
\label{se:2Dnum}

Let us consider finite dimensional subspaces $\V_h\subset\V$ and
$Q_h\subset Q$.
The approximation of~\eqref{eq:LSbfgpvar2D} consists in finding
$\bfu_h\in\V_h$ and $p_h\in Q_h$ such that
\begin{equation}
\left\{
\aligned
&(\eps\bfu_h,\bfv)+(\mu^{-1}\rot\bfu_h,\rot\bfv)-(\bfv,\curl p_h)=
(g,\rot\bfv)&&\forall\bfv\in\V_h\\
&-(\bfu_h,\curl q)+(\eps^{-1}\curl p_h,\curl q)=0&&\forall q\in Q_h
\endaligned
\right.
\label{eq:LSbfgpvar2Dh}
\end{equation}
and, correspondingly, the discrete eigenvalue problem we are interested in,
reads: find $\lambda_h\in\RE$ and $p_h\in Q_h$ with $p_h\ne0$ such that for
some $\bfu_h\in\V_h$ it holds
\begin{equation}
\left\{
\aligned
&(\eps\bfu_h,\bfv)+(\mu^{-1}\rot\bfu_h,\rot\bfv)-(\bfv,\curl p_h)=
\lambda_h(p_h,\rot\bfv)&&\forall\bfv\in\V_h\\
&-(\bfu_h,\curl q)+(\eps^{-1}\curl p_h,\curl q)=0&&\forall q\in Q_h.
\endaligned
\right.
\label{eq:LSbfgpvar2Deigh}
\end{equation}

We start our analysis of the discrete problem by discussing the convergence of
the solution of~\eqref{eq:LSbfgpvar2Dh} towards the solution
of~\eqref{eq:LSbfgpvar2D}. This is a standard result in the framework of
finite element Least-Squares approximations that follows from the coercivity
of the system recalled in Proposition~\ref{pr:ell2D}.

In the following theorem we recall the a priori error analysis that is a
standard consequence of C\'ea's lemma.

\begin{theorem}
Given $g\in Q$, let $(\bfu,p)$ be the solution of~\eqref{eq:LSbfgpvar2D} and
$(\bfu_h,p_h)$ the corresponding discrete solution of~\eqref{eq:LSbfgpvar2Dh}.
Then the following estimate holds true
\[
\|\bfu-\bfu_h\|_{\V}+\|p-p_h\|_Q\le C\inf_{(\bfv,q)\in\V_h\times Q_h}
\left(\|\bfu-\bfv\|_{\V}+\|p-q\|_Q\right)
\]
\label{th:cea}
\end{theorem}

In order to study the convergence of the eigensolutions
of~\eqref{eq:LSbfgpvar2Deigh} towards those of~\eqref{eq:LSbfgpvar2Deig} a
standard tool is the convergence in norm of the sequence of discrete solution
operators towards the continuous one. Indeed, in~\cite{bbg} it was shown that
for symmetric problems the convergence in norm is not only sufficient but also
necessary for the convergence of the eigensolutions.

Analogously to what we have done in the continuous case, let $T_h:Q\to Q$ be
the discrete solution operator which associated to $g\in Q$ the component
$p_h\in Q_h\subset Q$ of the solution of~\eqref{eq:LSbfgpvar2Dh}. A necessary
and sufficient condition for the convergence of our eigensolutions is the
existence of $\rho(h)$, tending to zero as $h$ goes to zero, such that
\begin{equation}
\|(T-T_h)g\|_Q\le\rho(h)\|g\|_Q.
\label{eq:cunif}
\end{equation}

The most natural choice of finite element spaces is to consider N\'ed\'elec
\emph{edge} elements for $\V_h$ and standard Lagrange \emph{nodal} elements for
$Q_h$, that is, for $k\ge0$,
\begin{equation}
\aligned
&\V_h=\{\bfv\in\V:\bfv|_K\in\Ned_k(K)\ \forall K\in\T_h\}\\
&Q_h=\{q\in Q:q|_K\in\Po_{k+1}(K)\ \forall K\in\T_h\},
\endaligned
\label{eq:spaces}
\end{equation}
where $\Po_k(K)$ is the space of polynomials on $K$ of degree not exceeding
$k$ and
\[
\Ned_k(K)=[\Po_k(K)]^2+\Po_k(K)[(x,y)]^\top.
\]
Other possible choices would involve different order of approximation for the
two spaces.

From the standard approximation properties of these spaces, and
assumption~\eqref{eq:epsmu}, Theorem~\ref{th:cea} gives
\[
\|(T-T_h)g\|_Q=\|p-p_h\|_Q\le C\inf_{(\bfv,q)\in\V_h\times Q_h}
\left(\|\bfu-\bfv\|_{\V}+\|p-q\|_Q\right).
\]

From Remark~\ref{re:neumann} we can take $s>1/2$ and proceed as follows.
\[
\aligned
\|\bfu-\bfv\|_{\V}^2&=\|\eps^{1/2}(\bfu-\bfv)\|_0^2
+\|\mu^{-1/2}\rot(\bfu-\bfv)\|_0^2\\
&\le C\overline\eps h^{2s}\|\bfu\|_s^2+
\|\mu^{-1/2}\rot(\bfu-\bfv)\|_0^2.
\endaligned
\]
We assume that $\mu$ is piecewise regular and that the mesh is compatible in
the sense that $\mu$ is smooth in each element, so that $\rot\bfu$ belongs to
$H^s(K)$ for each $K\in\T_h$. Then the estimate reads
\[
\aligned
\|\bfu-\bfv\|_{\V}^2&\le C\overline\eps h^{2s}\|\bfu\|_s^2+
\sum_{K\in\T_h}\|\mu^{-1/2}\rot(\bfu-\bfv)\|_{0,K}^2\\
&\le C\overline\eps h^{2s}\|\bfu\|_s^2+ C
\underline\mu^{-1}\sum_{K\in\T_h}h_K^{2s}\|\rot\bfu\|_{s,K}^2\\
&\le C\overline\eps h^{2s}\|\bfu\|_s^2+ C(\eps,\mu) h^{2s}\|g\|_Q.
\endaligned
\]
For the variable $p$ standard approximation properties give
\[
\|p-q\|_Q\le C\underline\eps^{-1/2}h^s\|p\|_{1+s},
\]
so that the final estimate reads
\begin{equation}
\|(T-T_h)g\|_Q\le Ch^s\|g\|_Q.
\label{eq:sestimate}
\end{equation}

Hence, we can conclude the convergence of the eigensolutions (and the absence
of spurious modes). In the next theorem we state this result, together with
the optimal rate of convergence of the eigenmodes.

\begin{remark}
The definition of the finite element spaces~\eqref{eq:spaces} considers a
balanced choice of the polynomial degrees (see also~\eqref{eq:sestimate}). For
the well posedness of the discrete problem, however, there are no
compatibility conditions between the spaces and more general choices could be
made.
\end{remark}

\begin{theorem}
For $k\ge0$, let $\V_h$ and $Q_h$ be as in the above definitions, and assume
that $\lambda$ is an eigenvalue of~\eqref{eq:LSbfgpvar2Deig} of multiplicity
$m$ with associated eigenspace $E$. Then there exist exactly $m$ eigenvalues
$\lambda_{1,h}\le\cdots\le\lambda_{m,h}$ converging to $\lambda$. Moreover,
let us denote by $E_h$ the space spanned by eigenfunctions associated with the
$m$ discrete eigenvalues. Then
\[
\aligned
&|\lambda-\lambda_{i,h}|\le C\epsilon(h)^2&&(i=1,\dots,m)\\
&\hat\delta(E,E_h)\le C\epsilon(h),
\endaligned
\]
where
\[
\epsilon(h)=\sup_{\substack{p\in E\\\|p\|_1=1}}\|(T-T_h)p\|_Q
\]
and $\hat\delta(A,B)$ denotes the gap between the subspaces $A$ and $B$ of $Q$.

\end{theorem}

\begin{remark}
The eigenspace $E$ and the space $E_h$ refer to the components $p$ and $p_h$
of the solutions of~\eqref{eq:LSbfgpvar2Deig} and~\eqref{eq:LSbfgpvar2Deigh},
respectively. However, the variables we are interested in when discussing the
eigenproblem associated with Maxwell's equations, are $\bfu$ and $\bfu_h$. On
the other hand, problems~~\eqref{eq:LSbfgpvar2Deig}
and~\eqref{eq:LSbfgpvar2Deigh} provide us also with the other (unique)
component of the solution which satisfies the following approximation
property. Let us denote by $F$ the space spanned by $\bfu$ associated with
$\lambda$ in~\eqref{eq:LSbfgpvar2Deig} and by $F_h$ the discrete space spanned
by the corresponding $\{\bfu_{1,h},\dots,\bfu_{m,h}\}$ obtained
from~\eqref{eq:LSbfgpvar2Deigh}. Then we have
\[
\hat\delta_{\rot}(F,F_h)\le C\epsilon(h),
\]
where $\hat\delta_{\rot}(A,B)$ denotes the gap between subspaces $A$ and $B$
of $\Horot$.

We also observe that if $E\subset H^{r_1+1}$ and $F\subset\bfH^{r_2}(\rot)$
then $\epsilon(h)=O(h^t)$ with $t=\min(k+1,r_1,r_2)$.
\end{remark}

One of the main messages that are conveyed when discussing Least-Squares
finite element methods, is that, thanks to the coercivity of the formulation,
any choice of finite element spaces is admissible, without the need of
satisfying a compatibility condition.

Several authors have tried to approximate the eigensolutions of the resonant
cavity by using standard \emph{nodal} elements. With this in mind, we might
think of choosing as $\V_h$ a space of Lagrange nodal elements in each
component and as $Q_h$ a space of Lagrange nodal elements as well. For
instance, an equal order approximation would involve the following spaces:
\[
\aligned
&\V_h=\{\bfv\in\bfH^1(\Omega)\cap\V:\bfv|_K\in(\Po_k{(K)})^2\ \forall
K\in\T_h\}\\
&Q_h=\{q\in Q:q|_K\in\Po_k{(K)}\ \forall K\in\T_h\}.
\endaligned
\]

In this case, the analogous of~\eqref{eq:sestimate} can be derived from
Theorem~\ref{th:cea} and the approximation properties of the finite element
spaces as follows
\begin{equation}
\aligned
\|(T-T_h)g\|_Q&=\|p-p_h\|_Q\le C\inf_{(\bfv,q)\in\V_h\times Q_h}
\left(\|\bfu-\bfv\|_{\V}+\|p-q\|_Q\right)\\
&\le C\left(\inf_{\bfv\in\V_h}\|\bfu-\bfv\|_{\V}+h^s\|p\|_{1+s}\right).
\endaligned
\label{eq:sestimaten}
\end{equation}
The estimate of the first term in the right hand side of~\eqref{eq:sestimaten}
is not as immediate as in the case of edge elements. Indeed, in the case of
nodal element, the best approximation in $\Hrot$ is not in general better than
the best approximation in $\bfH^1(\Omega)$. It follows
that~\eqref{eq:sestimaten} gives
\[
\|(T-T_h)g\|_Q\le Ch^s(\|\bfu\|_{1+s}+\|p\|_{1+s}).
\]
On the other hand, since $\eps\bfu=\curl p$, we get
\[
\|(T-T_h)g\|_Q\le Ch^s\|p\|_{2+s}
\]
and we can obtain the convergence in norm~\eqref{eq:cunif} if
Problem~\eqref{eq:bfgp2D} satisfies the a priori bound
\[
\|p\|_{2+s}\le C\|g\|_Q.
\]
From Remark~\ref{re:neumann} we know that this regularity holds true only in
very particular circumstances. In particular, we can prove the optimal
convergence of the eigensolutions (and the absence of spurious modes) in the
case of $\eps$ and $\mu$ constant, and when $\Omega$ is a square. More general
configurations are controversial from the theoretical point of view. On the
other hand, our numerical simulations presented in the next section show that
the method is pretty robust also in presence of strongly singular solutions.

\section{Numerical examples in two dimensions}
\label{se:2Dres}

Analogously to what was observed in Section~\ref{se:operatorform}, the
algebraic system associated with the discrete eigenvalue problem has the form
\[
\begin{pmatrix}
\sfA&\sfB^\top\\
\sfB&\sfC
\end{pmatrix}
\begin{pmatrix}
\sfx\\
\sfy
\end{pmatrix}
=\lambda
\begin{pmatrix}
\sfzero&\sfD\\
\sfzero&\sfzero
\end{pmatrix}
\begin{pmatrix}
\sfx\\
\sfy
\end{pmatrix},
\]
where the blocks of the matrices correspond to the pieces
in~\eqref{eq:LSbfgpvar2Deigh} according to the following mapping
\[
\aligned
&\sfA:&&(\eps\bfu_h,\bfv_h)+(\mu^{-1}\rot\bfu_h,\rot\bfv_h)\\
&\sfB:&&-(\bfu_h,\curl q_h)\\
&\sfC:&&(\eps^{-1}\curl p_h,\curl q_h)\\
&\sfD:&&(p_h,\rot\bfv_h).
\endaligned
\]

It turns out that this is a degenerate generalized eigenvalue problem. It is
out of the scope of this paper to discuss the best strategy for its
resolution. In our FEniCS code~\cite{dolfin,Fenics} we call the SLEPc
solver~\cite{slepc-users-manual}. Our options include the use of
mumps~\cite{Mumps} for dealing with the singular matrix on the right hand side
and shift and invert with shift equal to zero to compute the smallest
eigenvalues.

A crucial comment, which will become essential for the three dimensional
extension of our tests, refers to the gauge condition of the variables $p$ and
$q$. According to the variational formulation~\eqref{eq:LSbfgpvar2Deigh}, we
should impose on $p$ and $q$ the zero mean value condition after we multiply
them by $\mu$. This can be easily performed by adding a Lagrange multiplier at
the expense of increasing by one the dimension of the system.
The full variational formulation would read: find $\lambda_h\in\RE$ and
$p_h\in\widetilde Q_h$, with $p_h\ne0$, such that for some $\bfu_h\in\V_h$
and $\phi_h\in\RE$
\begin{equation}
\left\{
\aligned
&(\eps\bfu_h,\bfv)+(\mu^{-1}\rot\bfu_h,\rot\bfv)-(\bfv,\curl p_h)=
\lambda_h(p_h,\rot\bfv)&&\forall\bfv\in\V_h\\
&-(\bfu_h,\curl q)+(\eps^{-1}\curl p_h,\curl q)+(\phi_h,\mu q)=0&&\forall
q\in\widetilde Q_h\\
&(\mu p_h,\psi)=0&&\forall\psi\in\RE,
\endaligned
\right.
\label{eq:LSbfgpvar2Deigh-nomean}
\end{equation}
where $\widetilde Q_h$ is the same finite element space as
in~\eqref{eq:LSbfgpvar2Deigh} without imposing the zero mean value condition.

Clearly, problems~\eqref{eq:LSbfgpvar2Deigh}
and~\eqref{eq:LSbfgpvar2Deigh-nomean} are equivalent.

\begin{remark}

Similarly to what we observed at the end of Section~\ref{se:operatorform}, if
we solve the unconstrained problem~\eqref{eq:LSbfgpvar2Deigh} with $Q_h$
replaced by $\widetilde Q_h$ instead of~\eqref{eq:LSbfgpvar2Deigh-nomean},
the only difference is the introduction of an additional singular pencil
corresponding to the equation $\sfzero=\lambda\sfzero$ with $\bfu_h=\bfzero$
and $p$ constant.

\end{remark}

The first test that we present is a sanity check on the square $(0,\pi)^2$ and
$\mu=\eps=1$. It is well known that the eigenvalues are $m^2+n^2$ for
$m,n\ge0$ and $m+n>0$. Table~\ref{tb:edgesquare} confirms the optimal
convergence of the edge element approximation. Since the theoretical results
clearly indicate that edge elements are optimally convergent, we focus now on
a deeper investigation when nodal elements are used.

\begin{table}
\caption{ Edge elements on a uniform mesh of the square }
\begin{tabular}{r|rrr}
\hline
Exact&\multicolumn{3}{c}{Computed (rate)}\\[3pt]
\hline
1.00000 & 1.01090 & 1.00273 (2.00) & 1.00068 (2.00) \\
1.00000 & 1.01268 & 1.00316 (2.00) & 1.00079 (2.00) \\
2.00000 & 2.04063 & 2.01017 (2.00) & 2.00254 (2.00) \\
4.00000 & 4.11271 & 4.02792 (2.01) & 4.00696 (2.00) \\
4.00000 & 4.11276 & 4.02793 (2.01) & 4.00697 (2.00) \\
5.00000 & 5.14696 & 5.03683 (2.00) & 5.00922 (2.00) \\
5.00000 & 5.23986 & 5.05920 (2.02) & 5.01476 (2.00) \\
8.00000 & 8.49059 & 8.12382 (1.99) & 8.03103 (2.00) \\
9.00000 & 9.49795 & 9.12178 (2.03) & 9.03028 (2.01) \\
9.00000 & 9.51605 & 9.12582 (2.04) & 9.03127 (2.01) \\
\hline
Mesh & 1/16 & 1/32 & 1/64 \\
\hline
\end{tabular}
\label{tb:edgesquare}
\end{table}

In Tables~\ref{tb:nodalsquare} and~\ref{tb:nodalsquarenu} we perform the same
test with nodal (continuous piecewise linear) elements on a structured and
nonstructured mesh, respectively. It turns out that also in this case the
convergence is optimal and there are no spurious modes.

\begin{table}
\caption{ Nodal elements on a uniform mesh of the square }
\begin{tabular}{r|rrr}
\hline
Exact&\multicolumn{3}{c}{Computed (rate)}\\[3pt]
\hline
1.00000 & 1.00961 & 1.00240 (2.00) & 1.00060 (2.00) \\
1.00000 & 1.00963 & 1.00241 (2.00) & 1.00060 (2.00) \\
2.00000 & 2.03841 & 2.00962 (2.00) & 2.00241 (2.00) \\
4.00000 & 4.11637 & 4.02886 (2.01) & 4.00721 (2.00) \\
4.00000 & 4.11642 & 4.02886 (2.01) & 4.00721 (2.00) \\
5.00000 & 5.15263 & 5.03849 (1.99) & 5.00966 (1.99) \\
5.00000 & 5.23639 & 5.05860 (2.01) & 5.01464 (2.00) \\
8.00000 & 8.49203 & 8.12453 (1.98) & 8.03125 (1.99) \\
9.00000 & 9.56410 & 9.13748 (2.04) & 9.03424 (2.01) \\
9.00000 & 9.56410 & 9.13752 (2.04) & 9.03425 (2.01) \\
\hline
Mesh & 1/16 & 1/32 & 1/64 \\
\hline
\end{tabular}
\label{tb:nodalsquare}
\end{table}

\begin{table}
\caption{ Nodal elements on a nonuniform mesh of the square }
\begin{tabular}{r|rrr}
\hline
Exact&\multicolumn{3}{c}{Computed (rate)}\\[3pt]
\hline
1.00000 & 1.00478 & 1.00118 (2.01) & 1.00029 (2.01) \\
1.00000 & 1.00490 & 1.00120 (2.03) & 1.00030 (2.02) \\
2.00000 & 2.01564 & 2.00385 (2.02) & 2.00095 (2.01) \\
4.00000 & 4.05564 & 4.01362 (2.03) & 4.00336 (2.02) \\
4.00000 & 4.05739 & 4.01389 (2.05) & 4.00341 (2.03) \\
5.00000 & 5.08616 & 5.02078 (2.05) & 5.00512 (2.02) \\
5.00000 & 5.08816 & 5.02096 (2.07) & 5.00515 (2.03) \\
8.00000 & 8.21196 & 8.05101 (2.05) & 8.01255 (2.02) \\
9.00000 & 9.26651 & 9.06385 (2.06) & 9.01583 (2.01) \\
9.00000 & 9.27064 & 9.06476 (2.06) & 9.01591 (2.02) \\
\hline
Mesh & 1/16 & 1/32 & 1/64 \\
\hline
\end{tabular}
\label{tb:nodalsquarenu}
\end{table}

From~\eqref{eq:sestimaten} it was clear that the regularity of the solution is
essential in order to prove theoretically the convergence of the method. It is
then essential to verify the approximation behavior in presence of singular
solutions.
We start with the L-shaped domain obtained by the removing the upper right
square $(0,1)^2$ from the square $(-1,1)^2$, where we consider the reference
solution presented in~\cite{benchmax}.
Tables~\ref{tb:nodalL} and~\ref{tb:nodalLnu} show that also in this case the
nodal element approximation of our problem performs optimally and that no
spurious modes are present. It can be appreciated that singular solutions
present the expected lower rate of convergence.

\begin{table}
\caption{ Nodal elements on a uniform mesh of the L-shaped domain }
\label{tb:nodalL}
\begin{tabular}{r|rrr}
\hline
Exact&\multicolumn{3}{c}{Computed (rate)}\\[3pt]
\hline
1.47562 & 1.60421 & 1.52532 (1.37) & 1.49509 (1.35) \\
3.53403 & 3.56787 & 3.54233 (2.03) & 3.53609 (2.01) \\
9.86960 & 10.07466 & 9.92010 (2.02) & 9.88218 (2.01) \\
9.86960 & 10.07466 & 9.92010 (2.02) & 9.88218 (2.01) \\
11.38948 & 11.70401 & 11.46698 (2.02) & 11.40879 (2.00) \\
\hline
Mesh & 1/16 & 1/32 & 1/64 \\
\hline
\end{tabular}
\end{table}

\begin{table}
\caption{ Nodal elements on a nonuniform mesh of the L-shaped domain }
\label{tb:nodalLnu}
\begin{tabular}{r|rrrrr}
\hline
Exact&\multicolumn{5}{c}{Computed (rate)}\\[3pt]
\hline
1.47562 & 1.61215 & 1.53721 (1.15) & 1.51774 (0.55) & 1.49583 (1.06) & 1.48561 (1.02) \\
3.53403 & 3.55865 & 3.54038 (1.96) & 3.53568 (1.95) & 3.53440 (2.18) & 3.53414 (1.76) \\
9.86960 & 9.99824 & 9.90032 (2.07) & 9.87722 (2.01) & 9.87151 (2.00) & 9.87008 (2.01) \\
9.86960 & 10.00032 & 9.90054 (2.08) & 9.87739 (1.99) & 9.87152 (2.03) & 9.87008 (2.01) \\
11.38948 & 11.56604 & 11.43105 (2.09) & 11.39996 (1.99) & 11.39204 (2.03) & 11.39012 (1.99) \\
\hline
Mesh & 1/16 & 1/32 & 1/64 & 1/128 & 1/256 \\
\hline
\end{tabular}
\end{table}

We continue our investigations by considering more and more challenging
problems. The next one is the so called slit domain, that is the square
$(-1,1)^2$ from which we remove the segment $(0,1)\times\{0\}$. This is the
so called \emph{cracked domain} in~\cite{benchmax} that we use as a reference
solution. Also in this case, nodal elements behave optimally, in agreement
with the regularity of the solution, see Table~\ref{tb:nodalslit}.

\begin{table}
\caption{ Nodal elements on a uniform mesh of the slit domain }
\label{tb:nodalslit}
\begin{tabular}{r|rrr}
\hline
Exact&\multicolumn{3}{c}{Computed (rate)}\\[3pt]
\hline
1.03407 & 1.36608 & 1.19389 (1.05) & 1.11254 (1.03) \\
2.46740 & 2.48205 & 2.47105 (2.01) & 2.46831 (2.00) \\
4.04693 & 4.09184 & 4.05806 (2.01) & 4.04970 (2.00) \\
9.86960 & 10.07466 & 9.92010 (2.02) & 9.88218 (2.01) \\
9.86960 & 10.07466 & 9.92010 (2.02) & 9.88218 (2.01) \\
10.84485 & 11.11915 & 10.91246 (2.02) & 10.86170 (2.00) \\
12.26490 & 12.72449 & 12.43232 (1.46) & 12.36074 (0.80) \\
12.33701 & 12.96253 & 12.51124 (1.84) & 12.36438 (2.67) \\
19.73921 & 20.82226 & 20.00356 (2.03) & 19.80490 (2.01) \\
21.24411 & 22.81415 & 21.74694 (1.64) & 21.43012 (1.43) \\
\hline
Mesh & 1/16 & 1/32 & 1/64 \\
\hline
\end{tabular}
\end{table}

In order to make the solution even more singular, we consider the same slit
domain with mixed boundary conditions, that is we modify the boundary
conditions in the definition of the spaces $\V$ and $Q$ as follows: functions
in $\V$ have zero tangential component along the exterior boundary, while
functions in $Q$ are vanishing on the slit.

Since we do not have a reference solution in this case, we compare our results
with the eigenvalues computed with a standard Galerkin approximation of the
Laplace eigenvalue problem with mixed boundary conditions (Dirichlet on the
slit and Neumann on the rest of the boundary). Table~\ref{tb:nodalslitmixed}
shows that also in this case our solution is convergent and that no spurious
modes are present.

\begin{table}
\caption{ Nodal elements on a uniform mesh of the slit domain with mixed boundary conditions }
\label{tb:nodalslitmixed}
\begin{tabular}{r|rrr}
\hline
Rank&\multicolumn{3}{c}{Computed with standard Galerkin}\\[3pt]
\hline
1 & 1.27238 & 1.14957 & 1.09097  \\
2 & 2.48205 & 2.47105 & 2.46831  \\
3 & 4.09155 & 4.05803 & 4.04970  \\
4 & 5.00737 & 4.95283 & 4.93930  \\
5 & 11.11902 & 10.91244 & 10.86170  \\
6 & 12.72449 & 12.43232 & 12.35737  \\
7 & 12.93065 & 12.49669 & 12.36074  \\
8 & 22.78151 & 21.73189 & 21.42284  \\
9 & 23.23963 & 22.45620 & 22.26849  \\
10 & 25.15087 & 24.18459 & 23.95261  \\
\hline
Rank&\multicolumn{3}{c}{Computed with Least-Squares}\\[3pt]
\hline
1 & 1.09561 & 1.06413 & 1.04894  \\
2 & 2.47402 & 2.46905 & 2.46781  \\
3 & 4.07072 & 4.05292 & 4.04843  \\
4 & 4.97718 & 4.94538 & 4.93744  \\
5 & 10.99818 & 10.88309 & 10.85441  \\
6 & 12.56704 & 12.39424 & 12.33447  \\
7 & 12.71164 & 12.43105 & 12.35130  \\
8 & 22.25050 & 21.59440 & 21.38219  \\
9 & 22.75004 & 22.34091 & 22.24009  \\
10 & 24.57890 & 24.05065 & 23.91967  \\
\hline
Rank&\multicolumn{3}{c}{Difference (rate)}\\[3pt]
\hline
1 & 0.17677 & 0.08544 (1.05) & 0.04203 (1.02) \\
2 & 0.00803 & 0.00199 (2.01) & 0.00050 (2.00) \\
3 & 0.02084 & 0.00511 (2.03) & 0.00127 (2.01) \\
4 & 0.03019 & 0.00746 (2.02) & 0.00186 (2.00) \\
5 & 0.12084 & 0.02935 (2.04) & 0.00728 (2.01) \\
6 & 0.15745 & 0.03808 (2.05) & 0.02289 (0.73) \\
7 & 0.21901 & 0.06564 (1.74) & 0.00944 (2.80) \\
8 & 0.53101 & 0.13749 (1.95) & 0.04065 (1.76) \\
9 & 0.48960 & 0.11529 (2.09) & 0.02840 (2.02) \\
10 & 0.57197 & 0.13394 (2.09) & 0.03294 (2.02) \\
\hline
Mesh & 1/16 & 1/32 & 1/64 \\
\hline
\end{tabular}
\end{table}

Finally, we challenge our code with the computation of the problem with
different materials, that is jumping coefficients $\eps$ and $\mu$. In this
case, we compare with the standard \emph{curl curl} formulation discretized by
edge elements. We take a uniform mesh compatible with the jump of the
material. The test cases consider the square $(0,\pi)^2$ containing a
different material in the bottom-left square $(0,\pi/2)^2$.
Table~\ref{tb:epsilon} shows the comparison of the results when $\eps$ is
equal to $100$ on the different material and $1$ otherwise, with $\mu=1$
everywhere, while Table~\ref{tb:mu} deals with the case when $\eps=0$
everywhere and $\mu=1/100$ on the different material and $1$ otherwise. Also
in this case our method performs quite well.

\begin{table}
\caption{ Comparison of curl curl formulation (edge elements) and our formulation (nodal elements) on a uniform and fitted mesh of the square with jumping $\varepsilon$ }
\begin{tabular}{r|rrr}
\hline
Rank&\multicolumn{3}{c}{Computed with standard Galerkin}\\[3pt]
\hline
1 & 0.01294 & 0.01295 & 0.01295  \\
2 & 0.01425 & 0.01425 & 0.01425  \\
3 & 0.02579 & 0.02579 & 0.02579  \\
4 & 0.04612 & 0.04613 & 0.04613  \\
5 & 0.05126 & 0.05128 & 0.05129  \\
6 & 0.09252 & 0.09261 & 0.09264  \\
7 & 0.09407 & 0.09412 & 0.09413  \\
8 & 0.09971 & 0.09973 & 0.09973  \\
9 & 0.10740 & 0.10744 & 0.10746  \\
10 & 0.11554 & 0.11556 & 0.11557  \\
\hline
Rank&\multicolumn{3}{c}{Computed with Least-Squares}\\[3pt]
\hline
1 & 0.01483 & 0.01444 & 0.01391  \\
2 & 0.01602 & 0.01464 & 0.01433  \\
3 & 0.02677 & 0.02616 & 0.02596  \\
4 & 0.04732 & 0.04645 & 0.04622  \\
5 & 0.05505 & 0.05286 & 0.05205  \\
6 & 0.09655 & 0.09458 & 0.09359  \\
7 & 0.09726 & 0.09480 & 0.09433  \\
8 & 0.10286 & 0.10054 & 0.09995  \\
9 & 0.11148 & 0.10893 & 0.10808  \\
10 & 0.12236 & 0.11778 & 0.11642  \\
\hline
Rank&\multicolumn{3}{c}{Difference (rate)}\\[3pt]
\hline
1 & 0.00189 & 0.00149 (0.34) & 0.00096 (0.64) \\
2 & 0.00177 & 0.00039 (2.19) & 0.00007 (2.46) \\
3 & 0.00098 & 0.00037 (1.40) & 0.00016 (1.18) \\
4 & 0.00120 & 0.00031 (1.93) & 0.00009 (1.87) \\
5 & 0.00379 & 0.00158 (1.26) & 0.00075 (1.07) \\
6 & 0.00403 & 0.00197 (1.04) & 0.00094 (1.06) \\
7 & 0.00319 & 0.00068 (2.24) & 0.00020 (1.77) \\
8 & 0.00315 & 0.00082 (1.94) & 0.00022 (1.90) \\
9 & 0.00408 & 0.00149 (1.45) & 0.00062 (1.27) \\
10 & 0.00682 & 0.00222 (1.62) & 0.00085 (1.39) \\
\hline
Mesh & 1/64 & 1/128 & 1/256 \\
\hline
\end{tabular}
\label{tb:epsilon}
\end{table}

\begin{table}
\caption{ Comparison of curl curl formulation (edge elements) and our formulation (nodal elements) on a uniform and fitted mesh of the square with jumping $\mu$ }
\begin{tabular}{r|rrr}
\hline
Rank&\multicolumn{3}{c}{Computed with standard Galerkin}\\[3pt]
\hline
1 & 4.44469 & 4.44455 & 4.44451  \\
2 & 5.32703 & 5.32940 & 5.33000  \\
3 & 11.85091 & 11.85115 & 11.85121  \\
4 & 16.81898 & 16.83221 & 16.83553  \\
5 & 17.54526 & 17.56033 & 17.56410  \\
6 & 24.83957 & 24.82828 & 24.82545  \\
7 & 25.87595 & 25.90293 & 25.90970  \\
8 & 36.72957 & 36.78732 & 36.80175  \\
9 & 37.55359 & 37.62456 & 37.64233  \\
10 & 39.98800 & 39.97241 & 39.96839  \\
\hline
Rank&\multicolumn{3}{c}{Computed with Least-Squares}\\[3pt]
\hline
1 & 4.49865 & 4.47280 & 4.45740  \\
2 & 5.43938 & 5.37976 & 5.35001  \\
3 & 12.19870 & 12.03015 & 11.93174  \\
4 & 17.01187 & 16.90259 & 16.86498  \\
5 & 17.73776 & 17.62300 & 17.58675  \\
6 & 25.38670 & 25.04539 & 24.91860  \\
7 & 26.46349 & 26.16889 & 26.03923  \\
8 & 37.38132 & 36.96354 & 36.85391  \\
9 & 38.25990 & 37.81014 & 37.69398  \\
10 & 41.12399 & 40.40199 & 40.15998  \\
\hline
Rank&\multicolumn{3}{c}{Difference (rate)}\\[3pt]
\hline
1 & 0.05397 & 0.02826 (0.93) & 0.01288 (1.13) \\
2 & 0.11235 & 0.05035 (1.16) & 0.02001 (1.33) \\
3 & 0.34779 & 0.17900 (0.96) & 0.08053 (1.15) \\
4 & 0.19290 & 0.07038 (1.45) & 0.02946 (1.26) \\
5 & 0.19249 & 0.06267 (1.62) & 0.02265 (1.47) \\
6 & 0.54712 & 0.21712 (1.33) & 0.09315 (1.22) \\
7 & 0.58754 & 0.26595 (1.14) & 0.12953 (1.04) \\
8 & 0.65176 & 0.17622 (1.89) & 0.05216 (1.76) \\
9 & 0.70631 & 0.18558 (1.93) & 0.05164 (1.85) \\
10 & 1.13599 & 0.42958 (1.40) & 0.19159 (1.16) \\
\hline
Mesh & 1/64 & 1/128 & 1/256 \\
\hline
\end{tabular}
\label{tb:mu}
\end{table}

We conclude this section by showing some more numerical results which are not
completely covered by the theory. Namely we consider the situation when the
jump in the coefficient is not aligned with the mesh. In order to do so, we
consider the same geometry as in the previous examples reported in
Tables~\ref{tb:epsilon} and~\ref{tb:mu} with a nonuniform mesh of the domain
$\Omega$. It turns out that in this case the background mesh is not fitted
with the material discontinuities. The results are pretty much convincing also
in this case, as it can be seen from Table~\ref{tb:epsilonnu} in the case of
jumping $\eps$ and in Table~\ref{tb:munu} when $\mu$ jumps.

\begin{table}
\caption{ Comparison of curl curl formulation (edge elements) and our formulation (nodal elements) on a nonuniform and unfitted mesh of the square with jumping $\varepsilon$ }
\label{tb:epsilonnu}
\begin{tabular}{r|rrr}
\hline
Rank&\multicolumn{3}{c}{Computed with standard Galerkin}\\[3pt]
\hline
1 & 0.01306 & 0.01299 & 0.01298  \\
2 & 0.01450 & 0.01437 & 0.01431  \\
3 & 0.02607 & 0.02592 & 0.02585  \\
4 & 0.04625 & 0.04619 & 0.04616  \\
5 & 0.05163 & 0.05141 & 0.05138  \\
6 & 0.09324 & 0.09286 & 0.09282  \\
7 & 0.09584 & 0.09496 & 0.09455  \\
8 & 0.10070 & 0.10018 & 0.09993  \\
9 & 0.10921 & 0.10818 & 0.10786  \\
10 & 0.11742 & 0.11644 & 0.11603  \\
\hline
Rank&\multicolumn{3}{c}{Computed with Least-Squares}\\[3pt]
\hline
1 & 0.01566 & 0.01457 & 0.01407  \\
2 & 0.01876 & 0.01632 & 0.01528  \\
3 & 0.02909 & 0.02729 & 0.02653  \\
4 & 0.04754 & 0.04666 & 0.04635  \\
5 & 0.05450 & 0.05285 & 0.05226  \\
6 & 0.09631 & 0.09445 & 0.09377  \\
7 & 0.09978 & 0.09701 & 0.09569  \\
8 & 0.10723 & 0.10225 & 0.10066  \\
9 & 0.11593 & 0.11089 & 0.10925  \\
10 & 0.12434 & 0.11951 & 0.11753  \\
\hline
Rank&\multicolumn{3}{c}{Difference (rate)}\\[3pt]
\hline
1 & 0.00259 & 0.00158 (0.71) & 0.00108 (0.54) \\
2 & 0.00426 & 0.00195 (1.13) & 0.00097 (1.00) \\
3 & 0.00302 & 0.00137 (1.14) & 0.00068 (1.02) \\
4 & 0.00130 & 0.00047 (1.47) & 0.00019 (1.27) \\
5 & 0.00288 & 0.00144 (1.00) & 0.00088 (0.71) \\
6 & 0.00307 & 0.00159 (0.95) & 0.00095 (0.74) \\
7 & 0.00394 & 0.00205 (0.94) & 0.00114 (0.84) \\
8 & 0.00652 & 0.00207 (1.65) & 0.00073 (1.50) \\
9 & 0.00672 & 0.00271 (1.31) & 0.00139 (0.96) \\
10 & 0.00691 & 0.00307 (1.17) & 0.00150 (1.03) \\
\hline
Mesh & 1/64 & 1/128 & 1/256 \\
\hline
\end{tabular}
\end{table}

\begin{table}
\caption{ Comparison of curl curl formulation (edge elements) and our formulation (nodal elements) on a nonuniform and unfitted mesh of the square with jumping $\mu$ }
\label{tb:munu}
\begin{tabular}{r|rrr}
\hline
Rank&\multicolumn{3}{c}{Computed with standard Galerkin}\\[3pt]
\hline
1 & 4.33537 & 4.39753 & 4.41930  \\
2 & 5.19513 & 5.26094 & 5.29734  \\
3 & 11.57250 & 11.73322 & 11.78452  \\
4 & 16.44114 & 16.64371 & 16.73990  \\
5 & 17.10908 & 17.34339 & 17.46013  \\
6 & 24.25033 & 24.57767 & 24.68472  \\
7 & 25.25419 & 25.63192 & 25.76661  \\
8 & 35.92482 & 36.36569 & 36.59130  \\
9 & 36.70768 & 37.20499 & 37.42931  \\
10 & 38.93117 & 39.50523 & 39.73828  \\
\hline
Rank&\multicolumn{3}{c}{Computed with Least-Squares}\\[3pt]
\hline
1 & 4.34099 & 4.39904 & 4.41972  \\
2 & 5.20261 & 5.26322 & 5.29785  \\
3 & 11.60469 & 11.74193 & 11.78668  \\
4 & 16.49871 & 16.65967 & 16.74373  \\
5 & 17.17074 & 17.35981 & 17.46425  \\
6 & 24.37575 & 24.60994 & 24.69306  \\
7 & 25.38645 & 25.66805 & 25.77518  \\
8 & 36.18030 & 36.43259 & 36.60809  \\
9 & 36.97634 & 37.27509 & 37.44661  \\
10 & 39.23970 & 39.58511 & 39.75854  \\
\hline
Rank&\multicolumn{3}{c}{Difference (rate)}\\[3pt]
\hline
1 & 0.00562 & 0.00151 (1.90) & 0.00041 (1.87) \\
2 & 0.00749 & 0.00228 (1.71) & 0.00051 (2.15) \\
3 & 0.03219 & 0.00872 (1.88) & 0.00217 (2.01) \\
4 & 0.05757 & 0.01596 (1.85) & 0.00383 (2.06) \\
5 & 0.06166 & 0.01643 (1.91) & 0.00412 (1.99) \\
6 & 0.12542 & 0.03227 (1.96) & 0.00834 (1.95) \\
7 & 0.13227 & 0.03613 (1.87) & 0.00856 (2.08) \\
8 & 0.25548 & 0.06690 (1.93) & 0.01679 (1.99) \\
9 & 0.26866 & 0.07009 (1.94) & 0.01730 (2.02) \\
10 & 0.30854 & 0.07987 (1.95) & 0.02026 (1.98) \\
\hline
Mesh & 1/64 & 1/128 & 1/256 \\
\hline
\end{tabular}
\end{table}

\section{The three dimensional case}
\label{se:3D}

We start by writing a possible version of the three dimensional variational
problem associated with~\eqref{eq:LSbfgp} in the case when the uniqueness of
$\bfp$ is enforced by~\eqref{eq:gauge3D}.
Let $\V=\Hocurl$ and $\Qo=\Hcurl\cap\Hodivom$. The continuous problem reads:
given $\bfg\in\Hodivom$, find $(\bfu,\bfp)\in\V\times\Qo$ such that
\begin{equation}
\left\{
\aligned
&(\eps\bfu,\bfv)+(\mu^{-1}\curl\bfu,\curl\bfv)-(\bfv,\curl\bfp)=
(\bfg,\curl\bfv)&&\forall\bfv\in\V\\
&-(\bfu,\curl\bfq)+(\eps^{-1}\curl\bfp,\curl\bfq)=0&&\forall\bfq\in\Qo.
\endaligned
\right.
\label{eq:LSbfgpvar3D}
\end{equation}

The next proposition states the ellipticity of the bilinear form associated
with the above problem. We observe in particular that the norm of $\Qo$ is the
same as the $\Hcurl$ norm since the space contains divergence free
vectorfields.

\begin{proposition}

Let
\[
a:\left(\V\times\Qo\right)\times\left(\V\times\Qo\right)\to\RE
\]
be the bilinear form associated with the
formulation~\eqref{eq:LSbfgpvar3D}, that is
\[
a(\bfu,\bfp;\bfv,\bfq)=(\mu^{-1}\curl\bfu,\curl\bfv)
+(\eps^{1/2}\bfu-\eps^{-1/2}\curl\bfp,\eps^{1/2}\bfv-\eps^{-1/2}\curl\bfq).
\]
Then there exists $\alpha>0$ such that
\[
a(\bfv,\bfq;\bfv,\bfq)\ge\alpha\left(\|\bfv\|^2_{\curl}+\|\bfq\|_{\curl}^2\right).
\]
\label{pr:ell3D}

\end{proposition}

\begin{proof}

The proof is obtaining by extending the two dimensional proof of
Proposition~\ref{pr:ell2D}. The only critical point is the estimate of
$\|\bfq\|_0$ in terms of $\|\curl\bfq\|_0$ which can be performed thanks to
the following Friedrichs inequality (see, for instance~\cite[Cor.\
3.16]{amrouche}): there exists a constant $C_F$ such that
\[
\|\bfq\|_0\le C_F\|\curl\bfq\|_0\qquad\forall\bfq\in\Qo.
\]

\end{proof}

From the ellipticity we deduce from Lax--Milgram lemma that
problem~\eqref{eq:LSbfgpvar3D} admits a unique solution with the stability
estimate
\[
\|\bfu\|_{\curl}+\|\bfp\|_{\curl}\le C\|\bfg\|_0.
\]

As already observed, dealing with the approximation of $\Qo$ is not easy, so
that we propose the following formulation which imposes the divergence free
condition with an additional Lagrange multiplier $\phi\in\Hu\cap\Ldo$.

Let $\V=\Hocurl$, $\Q=\Hcurl$, and $\W=\Hu\cap\Ldo$. We consider the following
problem: given $\bfg\in\Hodivom$, find $(\bfu,\bfp,\phi)\in\V\times\Q\times\W$
such that
\begin{equation}
\left\{
\aligned
&(\eps\bfu,\bfv)+(\mu^{-1}\curl\bfu,\curl\bfv)-(\bfv,\curl\bfp)=
(\bfg,\curl\bfv)&&\forall\bfv\in\V\\
&-(\bfu,\curl\bfq)+(\eps^{-1}\curl\bfp,\curl\bfq)+(\mu\bfq,\grad\phi)=0&&\forall\bfq\in\Q\\
&(\mu\bfp,\grad\psi)=0&&\forall\psi\in\W.
\endaligned
\right.
\label{eq:LSbfgpvarmul3D}
\end{equation}

The following proposition states the equivalence of
Problems~\eqref{eq:LSbfgpvarmul3D} and~\eqref{eq:LSbfgpvar3D}.

\begin{proposition}

Let $(\bfu,\bfp)$ be a solution of~\eqref{eq:LSbfgpvar3D}, then
$(\bfu,\bfp,\phi)$ solves~\eqref{eq:LSbfgpvarmul3D} with $\phi=0$.
Conversely, if $(\bfu,\bfp,\phi)$ solves~\eqref{eq:LSbfgpvarmul3D}, then
$\bfp$ is in $\Qo$ and $(\bfu,\bfp)$ is a solution of~\eqref{eq:LSbfgpvar3D}.
\label{pr:kernel}
\end{proposition}

\begin{proof}

If $(\bfu,\bfp)$ is a solution of~\eqref{eq:LSbfgpvar3D} and $\phi=0$, then
clearly the first two equations of~\eqref{eq:LSbfgpvarmul3D} are satisfied.
Moreover, if $\bfp$ belongs to $\Qo$ then it also satisfies the third equation.

Vice versa, the third equation of~\eqref{eq:LSbfgpvarmul3D} implies that
$\bfp$ belongs to $\Qo$. Indeed, by integration by parts we formally have
\[
(\mu\bfp,\grad\psi)=-(\div(\mu\bfp),\psi)
+\langle(\mu\bfp)\cdot\bfn,\psi\rangle_{\partial\Omega}.
\]
Taking $\psi$ in $\mathscr{D}(\Omega)$ we obtain that $\mu\bfp$ is in 
$\Hdiv$ and $\div(\mu\bfp)=0$ in $\Omega$. Then, a generic $\psi\in\W$ gives
the result on the boundary that the trace of $(\mu\bfp)\cdot\bfn$ is vanishing.

It remains to show that $\phi$ is equal to zero so that the second equation
of~\eqref{eq:LSbfgpvarmul3D} corresponds to the second equation
of~\eqref{eq:LSbfgpvar3D}. Taking $\bfq=\grad\phi$ we have
$\|\mu^{1/2}\grad\phi\|_0=0$, that is $\grad\phi=0$, so that $\phi=0$ because
it is in $\Ldo$.

\end{proof}

For completeness, we give a direct proof of the existence and uniqueness of
the solution of~\eqref{eq:LSbfgpvarmul3D}.

\begin{proposition}

Let $\bfg$ be in $\Hodivom$, then there exists a unique solution
$(\bfu,\bfp,\phi)\in\V\times\Q\times\W$ of Problem~\eqref{eq:LSbfgpvarmul3D}
that satisfies $\phi=0$ and
\[
\|\bfu\|_{\curl}+\|\bfp\|_{\curl}\le C\|\bfg\|_0.
\]
\label{pr:ell3D_LM}
\end{proposition}

\begin{proof}

By extending to $\V\times\Q$ the bilinear form $a$ defined in the statement of
Proposition~\ref{pr:ell3D}, we can rewrite Problem~\eqref{eq:LSbfgpvarmul3D}
as: find $(\bfu,\bfp,\phi)\in\V\times\Q\times\W$
such that
\[
\left\{
\aligned
&a(\bfu,\bfp;\bfv,\bfq)+(\mu\bfq,\grad\phi)=(\bfg,\curl\bfv)&&\forall(\bfv,\bfq)\in\V\times\Q\\
&(\mu\bfp,\grad\psi)=0&&\forall\psi\in\W.
\endaligned
\right.
\]
This is a typical saddle point problem for which we show the \emph{ellipticity
in the kernel} and the \emph{inf-sup condition}.

The kernel is defined as
\[
K=\{(\bfv,\bfq)\in\V\times\Q:(\mu\bfq,\grad\psi)=0\ \forall\psi\in\W\},
\]
that is, $K=\V\times\Qo$ (see the proof of Proposition~\ref{pr:kernel}).
Hence, the ellipticity in the kernel follows from the ellipticity result
proved in Proposition~\ref{pr:ell3D}.

The inf-sup condition
\[
\inf_{\psi\in\W}\sup_{\bfq\in\Q}\frac{(\mu\bfq,\grad\psi)}{\|\bfq\|_{\curl}\|\psi\|_1}\ge\beta_0
\]
follows in a standard way by observing that, given $\psi\in\W$,
$\bfq=\grad\psi$ belongs to $\Q$ and
\[
\frac{(\mu\bfq,\grad\psi)}{\|\bfq\|_{\curl}}\ge\beta_0\|\psi\|_1,
\]
where $\beta_0$ depends on the Poincar\'e constant $C_P$.

Hence, the existence and stability follows from the standard results of mixed
problems (see~\cite{bbf}, for instance) and $\phi$ is zero as observed in the
proof of Proposition~\ref{pr:kernel}.

\end{proof}

Let us introduce a general finite element discretization of
Problem~\eqref{eq:LSbfgpvarmul3D} by considering finite dimensional subspaces
$\V_h\subset\V$, $\Q_h\subset\Q$ and $\W_h\subset\W$. Then the discrete
counterpart of~\eqref{eq:LSbfgpvarmul3D} reads: given $\bfg\in\Hodivom$, find
$(\bfu_h,\bfp_h,\phi_h)\in\V_h\times\Q_h\times\W_h$ such that
\begin{equation}
\left\{
\aligned
&(\eps\bfu_h,\bfv)+(\mu^{-1}\curl\bfu_h,\curl\bfv)-(\bfv,\curl\bfp_h)=
(\bfg,\curl\bfv)&&\forall\bfv\in\V_h\\
&-(\bfu_h,\curl\bfq)+(\eps^{-1}\curl\bfp_h,\curl\bfq)+(\mu\bfq,\grad\phi_h)=0
&&\forall\bfq\in\Q_h\\
&(\mu\bfp_h,\grad\psi)=0&&\forall\psi\in\W_h.
\endaligned
\right.
\label{eq:LSbfgpvarmul3Dh}
\end{equation}
The analysis of existence and uniqueness of the solution of this problem can be
performed following the same lines of the proof of
Proposition~\ref{pr:ell3D_LM}. The crucial point is to show the ellipticity of
the bilinear form $a$ on the discrete kernel. This is not true for any choice
of discrete spaces, therefore we consider N\'ed\'elec edge elements for both
$\V_h$ and $\Q_h$ and standard Lagrange nodal elements for $\W_h$. For
$K\in\T_h$ and $k\ge0$, let $\mathcal{N}_k(K)$ be the following space:
\[
\mathcal{N}_k(K)=[\mathcal{P}_k(K)]^3+\mathcal{P}_k(K)[(x,y,z)]^\top,
\]
then we introduce the following discrete spaces
\begin{equation}
\aligned
&\V_h=\{\bfv\in\V: \bfv|_K\in\mathcal{N}_k(K)\ \forall K\in\T_h\}\\  
&\Q_h=\{\bfq\in\Q: \bfq|_K\in\mathcal{N}_k(K)\ \forall K\in\T_h\}\\
&\W_h=\{\psi\in\W: \psi|_K\in\mathcal{P}_k(K)\ \forall K\in\T_h\}.
\endaligned
\label{eq:spaces3D}
\end{equation}
With this choice for the finite element spaces we can show the following
proposition:
\begin{proposition}

Let $\bfg$ be in $\Hodivom$, then there exists a unique solution
$(\bfu_h,\bfp_h,\phi_h)\in\V_h\times\Q_h\times\W_h$ of
Problem~\eqref{eq:LSbfgpvarmul3Dh} that satisfies $\phi_h=0$ and
\[
\|\bfu_h\|_{\curl}+\|\bfp_h\|_{\curl}\le C\|\bfg\|_0.
\]
Moreover the following error estimates holds true:
\[
\|\bfu-\bfu_h\|_{\curl}+\|\bfp-\bfp_h\|_{\curl}\le 
C\inf_{(\bfv,\bfq)\in\V_h\times\Q_h}
\left(\|\bfu-\bfv\|_{\curl}+\|\bfp-\bfq\|_{\curl}\right).
\]
\label{pr:apriori3D}
\end{proposition}
\begin{proof}
Using the bilinear form $a$ defined in Proposition~\ref{pr:ell3D} we can
write Problem~\eqref{eq:LSbfgpvarmul3Dh} as a saddle point problem as follows:
find $(\bfu_h,\bfp_h,\phi_h)\in\V_h\times\Q_h\times\W_h$ such that
\begin{equation}
\left\{
\aligned
&a(\bfu_h,\bfp_h;\bfv,\bfq)+(\mu\bfq,\grad\phi_h)=(\bfg,\curl\bfv)
&&\forall(\bfv,\bfq)\in\V_h\times\Q_h\\
&(\mu\bfp_h,\grad\psi)=0&&\forall\psi\in\W_h.
\endaligned
\right.
\label{eq:mixedh}
\end{equation}
The discrete kernel is given by
\[
\mathbb{K}_h=\{(\bfv_h,\bfq_h)\in\V_h\times\Q_h: (\mu\bfq_h,\grad\psi)=0\
\forall\psi\in\W_h\}.
\]
Then the discrete ellipticity in the kernel holds true thanks to the fact that
for functions $\bfq_h\in\Q_h$ such that $(\mu\bfq_h,\grad\psi)=0$
$\forall\psi\in\W_h$ the following Friedrichs inequality holds true
(see~\cite[Prop.~4.6]{abdg}) for a suitable constant $\tilde{C}_F$ not
depending on $h$
\[
\|\bfq_h\|_0\le \tilde{C}_F\|\curl\bfq_h\|_0.
\]
Recalling that $\grad\W_h\subset\Q_h$, we have also that the discrete infsup
condition
\[
\inf_{\psi\in\W_h}\sup_{\bfq\in\Q_h}
\frac{(\mu\bfq,\grad\psi)}{\|\bfq\|_{\curl}\|\psi\|_1}\ge \beta_1
\]
with $\beta_1>0$ depending on the Poincar\'e constant but not on $h$.
Indeed, given $\psi\in\W_h$ it is enough to take $\bfq=\grad\psi\in\V_h$.
Moreover, we observe that taking $(\bfv,\bfq)=(\boldsymbol{0},\grad\phi_h)$
in the first equation in~\eqref{eq:mixedh} we get
$(\mu\grad\phi_h,\grad\phi_h=0$ which implies that $\phi_h=0$ since it has zero
mean value.

Hence by the theory on the approximation of saddle point problems, (see, e.g.,
\cite{bbf}), we obtain existence, uniqueness and stability of the solution
of~\eqref{eq:mixedh} together with the error estimate.

\end{proof}

We can then write the eigenvalue problems associated
with~\eqref{eq:LSbfgpvar3D} and~\eqref{eq:LSbfgpvarmul3Dh} and the
corresponding continuous and discrete solution operators.

We recall that $\V=\Hocurl$, $\Q=\Hcurl$, and $\W=\Hu\cap\Ldo$. The continuous
eigenvalue problem reads: find $\lambda\in\RE$ such that for a non-vanishing
$\bfp\in\Q$ there exists $(\bfu,\phi)\in\V\times\W$
such that
\begin{equation}
\left\{
\aligned
&(\eps\bfu,\bfv)+(\mu^{-1}\curl\bfu,\curl\bfv)-(\bfv,\curl\bfp)=\lambda
(\bfp,\curl\bfv)&&\forall\bfv\in\V\\
&-(\bfu,\curl\bfq)+(\eps^{-1}\curl\bfp,\curl\bfq)+(\mu\bfq,\grad\phi)=0&&\forall\bfq\in\Q\\
&(\mu\bfp,\grad\psi)=0&&\forall\psi\in\W.
\endaligned
\right.
\label{eq:LSbfgpvarmul3Deig}
\end{equation}

The corresponding solution operator $T:\Q\to\Q$ is defined as $T\bfg=\bfp$
where $\bfp$ is the second component of the solution
to~\eqref{eq:LSbfgpvarmul3D}.

With the finite element spaces defined in~\eqref{eq:spaces3D}, the discrete
version of~\eqref{eq:LSbfgpvarmul3Deig} reads: find $\lambda_h\in\RE$ such that for a non-vanishing
$\bfp_h\in\Q_h$ there exists $(\bfu_h,\phi_h)\in\V_h\times\W_h$
such that
\begin{equation}
\left\{
\aligned
&(\eps\bfu_h,\bfv)+(\mu^{-1}\curl\bfu_h,\curl\bfv)-(\bfv,\curl\bfp_h)=\lambda_h
(\bfp_h,\curl\bfv)&&\forall\bfv\in\V_h\\
&-(\bfu_h,\curl\bfq)+(\eps^{-1}\curl\bfp_h,\curl\bfq)+(\mu\bfq,\grad\phi_h)=0&&\forall\bfq\in\Q_h\\
&(\mu\bfp_h,\grad\psi)=0&&\forall\psi\in\W_h
\endaligned
\right.
\label{eq:LSbfgpvarmul3Deigh}
\end{equation}
and the corresponding solution operator $T_h:\Q\to\Q_h\subset\Q$ is defined as
$T_h\bfg=\bfp_h$, where $\bfp_h$ is the second component of the solution
to~\eqref{eq:LSbfgpvarmul3Dh}.

From the a priori estimates proved in Proposition~\ref{pr:apriori3D}, we then
obtain the convergence of the discrete eigenmodes to the continuous one along
the lines of the classical Babu\v ska--Osborn theory.

\begin{theorem}
For $k\ge0$, let $\V_h$, $\Q_h$, and $W_h$ as defined in~\eqref{eq:spaces3D};
assume that $\lambda$ is an eigenvalue of~\eqref{eq:LSbfgpvarmul3Deig} of
multiplicity $m$ with associated eigenspace $E$. Then there exist exactly $m$
eigenvalues $\lambda_{1,h}\le\cdots\le\lambda_{m,h}$
of~\eqref{eq:LSbfgpvarmul3Deigh} converging to $\lambda$. Moreover, let us
denote by $E_h$ the space spanned by eigenfunctions associated with the
$m$ discrete eigenvalues. Then
\[
\aligned
&|\lambda-\lambda_{i,h}|\le C\epsilon(h)^2&&(i=1,\dots,m)\\
&\hat\delta(E,E_h)\le C\epsilon(h),
\endaligned
\]
where
\[
\epsilon(h)=\sup_{\substack{\bfp\in E\\\|\bfp\|_1=1}}\|(T-T_h)\bfp\|_{\Q}
\]
and $\hat\delta(A,B)$ denotes the gap between the subspaces $A$ and $B$ of $\Q$.

\end{theorem}

\section{Numerical examples in three dimensions}
\label{se:3Dres}

We conclude this paper with some three dimensional numerical results.
We limit ourselves to the results on the cube $(0,\pi)^3$ where the exact
results are known analytically.

Table~\ref{tb:edgecube} shows the results of the computation performed with
formulation~\eqref{eq:LSbfgpvarmul3Deigh} and confirms the good behavior of
the scheme.

\begin{table}
\caption{ Edge elements on a uniform mesh of the cube }
\label{tb:edgecube}
\begin{tabular}{r|rrr}
\hline
Exact&\multicolumn{3}{c}{Computed (rate)}\\[3pt]
\hline
2.00000 & 2.12293 & 2.03055 (2.01) & 2.00762 (2.00) \\
2.00000 & 2.12749 & 2.03140 (2.02) & 2.00782 (2.01) \\
2.00000 & 2.12749 & 2.03140 (2.02) & 2.00782 (2.01) \\
3.00000 & 3.25478 & 3.06315 (2.01) & 3.01576 (2.00) \\
3.00000 & 3.25478 & 3.06315 (2.01) & 3.01576 (2.00) \\
5.00000 & 5.50030 & 5.12147 (2.04) & 5.03016 (2.01) \\
5.00000 & 5.50930 & 5.12395 (2.04) & 5.03079 (2.01) \\
5.00000 & 5.50930 & 5.12395 (2.04) & 5.03079 (2.01) \\
5.00000 & 5.71701 & 5.17116 (2.07) & 5.04231 (2.02) \\
5.00000 & 5.71701 & 5.17116 (2.07) & 5.04231 (2.02) \\
\hline
Mesh & 1/8 & 1/16 & 1/32 \\
\hline
\end{tabular}
\end{table}

Finally, we show some promising results related to the nodal approximation of
our problem. Indeed, we are considering the following two field formulation:
find $\lambda_h\in\RE$ such that for a non-vanishing $\bfp_h\in\Q_h$ there
exists $\bfu_h\in\V_h$ such that
\[
\left\{
\aligned
&(\eps\bfu_h,\bfv)+(\mu^{-1}\curl\bfu_h,\curl\bfv)-(\bfv,\curl\bfp_h)=\lambda_h
(\bfp_h,\curl\bfv)&&\forall\bfv\in\V_h\\
&-(\bfu_h,\curl\bfq)+(\eps^{-1}\curl\bfp_h,\curl\bfq)=0&&\forall\bfq\in\Q_h,
\endaligned
\right.
\]
where the spaces $\V_h$ and $\Q_h$ are both consisting of continiuous
piecewise linear elements in each component. No theory is available in this
case and in general the obtained results can be wrong; in particular no gauge
condition is imposed on $\bfp$ so that there is no uniqueness in the case of
the source problem. Nevertheless, in the case of the cube we obtain the nice
and optimal results reported in Table~\ref{tb:nodalcube} in the case of a
uniform mesh and in Table~\ref{tb:nodalcubenu} when a nonuniform mesh is
considered.

\begin{table}
\caption{ Nodal elements on a uniform mesh of the cube }
\label{tb:nodalcube}
\begin{tabular}{r|rrr}
\hline
Exact&\multicolumn{3}{c}{Computed (rate)}\\[3pt]
\hline
2.00000 & 2.12841 & 2.03172 (2.02) & 2.00792 (2.00) \\
2.00000 & 2.12841 & 2.03172 (2.02) & 2.00792 (2.00) \\
2.00000 & 2.13005 & 2.03210 (2.02) & 2.00801 (2.00) \\
3.00000 & 3.27704 & 3.06924 (2.00) & 3.01734 (2.00) \\
3.00000 & 3.27704 & 3.06924 (2.00) & 3.01734 (2.00) \\
5.00000 & 5.56075 & 5.13631 (2.04) & 5.03393 (2.01) \\
5.00000 & 5.56516 & 5.13703 (2.04) & 5.03411 (2.01) \\
5.00000 & 5.56516 & 5.13703 (2.04) & 5.03411 (2.01) \\
5.00000 & 5.85266 & 5.20337 (2.07) & 5.05035 (2.01) \\
5.00000 & 5.86162 & 5.20337 (2.08) & 5.05035 (2.01) \\
\hline
Mesh & 1/8 & 1/16 & 1/32 \\
\hline
\end{tabular}
\end{table}

\begin{table}
\caption{ Nodal elements on a nonuniform mesh of the cube }
\label{tb:nodalcubenu}
\begin{tabular}{r|rrr}
\hline
Exact&\multicolumn{3}{c}{Computed (rate)}\\[3pt]
\hline
2.00000 & 2.39482 & 2.10329 (1.93) & 2.02523 (2.03) \\
2.00000 & 2.41541 & 2.10486 (1.99) & 2.02538 (2.05) \\
2.00000 & 2.42126 & 2.10563 (2.00) & 2.02546 (2.05) \\
3.00000 & 3.99333 & 3.22212 (2.16) & 3.05435 (2.03) \\
3.00000 & 4.00258 & 3.22486 (2.16) & 3.05464 (2.04) \\
5.00000 & 7.58994 & 5.61356 (2.08) & 5.14744 (2.06) \\
5.00000 & 7.64984 & 5.62135 (2.09) & 5.14764 (2.07) \\
5.00000 & 7.75720 & 5.62405 (2.14) & 5.14818 (2.07) \\
5.00000 & 7.80063 & 5.63059 (2.15) & 5.14861 (2.09) \\
5.00000 & 7.86660 & 5.63186 (2.18) & 5.14964 (2.08) \\
\hline
Mesh & 1/8 & 1/16 & 1/32 \\
\hline
\end{tabular}
\end{table}

\section*{Acknowledgments}
The author Lucia Gastaldi is member of INdAM Research group GNCS and she is
partially supported by PRIN/MIUR and by IMATI/CNR


\end{document}